\documentclass[12pt]{amsart}
\usepackage{amssymb,latexsym, amsmath, amscd, array, graphicx}

\swapnumbers
\numberwithin{equation}{section}

\theoremstyle{plain}

\newtheorem{thm}{Theorem}[section]
\newtheorem{theorem}[thm]{Theorem}

\newtheorem{lemma}[thm]{Lemma}

\newtheorem{prop}[thm]{Proposition}
\newtheorem{cor}[thm]{Corollary}

\newcommand\theoref{Theorem~\ref}
\newcommand\lemref{Lemma~\ref}
\newcommand\propref{Proposition~\ref}
\newcommand\corref{Corollary~\ref}
\newcommand\defref{Definition~\ref}

\newcommand\Zp{\Z_p}

\theoremstyle{definition}

\newtheorem{definition}[thm]{Definition}

\newtheorem{rem}[thm]{Remark}
\newtheorem{remark}[thm]{Remark}

\newtheorem{ex}[thm]{Example}

\newtheorem{question}[thm]{Question}

\DeclareMathOperator{\sys}{{\rm sys}}
\DeclareMathOperator{\sysh}{{\rm sysh}}

\DeclareMathOperator{\syscat}{{\rm cat_{sys}}}
\DeclareMathOperator{\cat}{{\mbox{\rm cat$_{\rm LS}$}}}

\def\Im{\protect\operatorname{Im}}

\def\Hom{\protect\operatorname{Hom}}

\def\wgt{\protect\operatorname{wgt}}
\def\cwgt{\protect\operatorname{cwgt}}
\def\swgt{\protect\operatorname{swgt}}
\def\cd{\protect\operatorname{cd}}

\newcommand \pa[2]{\frac{\partial #1}{\partial #2}}

\def\gf{\varphi}
\def\ga{\alpha}
\def\gb{\beta}

\def\scr{\mathcal}
\def\A{{\scr A}}
\def\B{{\scr B}}
\def\O{{\scr O}}
\def\C{{\mathbb C}}
\def\Z{{\mathbb Z}}

\def\R{{\mathbb R}}
\def\N{{\mathbb N}}

\def\ra{{\R A}}

\def\zp{\Z[\pi]}

\def\1{\hbox{\rm\rlap {1}\hskip.03in{\rom I}}}
\def\Bbbone{{\rm1\mathchoice{\kern-0.25em}{\kern-0.25em}
{\kern-0.2em}{\kern-0.2em}I}}

\def\pa{\partial}

\def\wt{\widetilde}

\def\m{\medskip}
\def\ov{\overline}

\def\la{\langle}
\def\ra{\rangle}

\long\def\forget#1\forgotten{} %

\newcommand{\gmetric}{{\mathcal G}}
\newcommand{\regular}{{B}}
\DeclareMathOperator{\stsys}{{\rm stsys}}
\DeclareMathOperator{\pisys}{{\rm sys}\pi}
\DeclareMathOperator{\vol}{{\rm vol}}

\newcommand\ver[1]{\marginpar{\tiny Changed in Ver \VER}}

\newcommand\T {{\mathbb T}}
\newcommand\AJ {{\mathcal A}}
\DeclareMathOperator{\fmanifold}{{[\overline{\it F}_{\it \! M}]}}
\newcommand\manbar {{\overline{{M}}}}
\newcommand{\surface}{\Sigma}

\date{\today}
\begin{document}

\title[Small values of LS and systolic categories for manifolds]{Small
values of Lusternik-Schnirelmann and systolic categories for
manifolds}

\author[A.~Dranishnikov]{Alexander N. Dranishnikov$^{1}$} %
\thanks{$^{1}$Supported by NSF, grant DMS-0604494}

\author[M.~Katz]{Mikhail G. Katz$^{2}$} %
\thanks{$^{2}$Supported by the Israel Science Foundation (grants
no.~84/03 and 1294/06)}

\author[Yu.~Rudyak]{Yuli B. Rudyak$^{3}$}%
\thanks{$^{3}$Supported by NSF, grant 0406311}

\address{Alexander N. Dranishnikov, Department of Mathematics, University
of Florida, 358 Little Hall, Gainesville, FL 32611-8105, USA}
\email{dranish@math.ufl.edu}

\address{Mikhail G. Katz, Department of Mathematics, Bar Ilan University,
Ramat Gan 52900 Israel}
\email{katzmik@math.biu.ac.il}

\address{Yuli B. Rudyak, Department of Mathematics, University
of Florida, 358 Little Hall, Gainesville, FL 32611-8105, USA}
\email{rudyak@math.ufl.edu}



\begin{abstract}
We prove that manifolds of Lusternik-Schnirelmann category~$2$
necessarily have free fundamental group.  We thus settle a 1992
conjecture of Gomez-Larra\~naga and Gonzalez-Acu\~na, by generalizing
their result in dimension~$3$, to all higher dimensions.  We examine
its ramifications in systolic topology, and provide a sufficient
condition for ensuring a lower bound of~$3$ for systolic category.
\end{abstract}

\maketitle
\tableofcontents

\section{Introduction}

We follow the normalization of the Lusternik-Schnirelmann category (LS
category) used in the recent monograph \cite{CLOT} (see
Section~\ref{three} for a definition).  Spaces of LS category~$0$ are
contractible, while a closed manifold of LS category~$1$ is homotopy
equivalent (and hence homeomorphic) to a sphere.

The characterization of closed manifolds of LS category~$2$ was
initiated in 1992 by J.~Gomez-Larra\~naga and F.~Gonzalez-Acu\~na
\cite{GoGo} (see also~\cite{OR}), who proved the following result on
closed manifolds~$M$ of dimension~$3$.  Namely, the fundamental group
of~$M$ is free and nontrivial if and only if its LS category is~$2$.
Furthermore, they conjectured that the fundamental group of every
closed~$n$-manifold,~$n\ge 3$, of LS category~$2$ is necessarily
free~\cite[Remark, p.~797]{GoGo}.  Our interest in this natural
problem was stimulated in part by recent work on the comparison of the
LS category and the systolic category \cite{KR1, KR2, SGT}, which was
inspired, in turn, by M.~Gromov's systolic inequalities \cite{Gr1,
Gr2, Gr3, Gr4}.

In the present text we prove this 1992 conjecture.  Recall that all
closed surfaces different from~$S^2$ are of LS category 2.

\begin{theorem}
\label{11}
A closed connected manifold of LS category~$2$ either is a surface, or
has free fundamental group.
\end{theorem}

\begin{cor}
\label{12b}
Every manifold~$M^n, n\geq 3$, with non-free fundamental group
satisfies~$\cat(M)\geq 3$.
\end{cor}

\begin{theorem}
\label{t:realization}
Given a finitely presented group~$\pi$ and non-negative integer
numbers~$k, l$, there exists a closed manifold~$M$ such
that~$\pi_1(M)=\pi$, while~$\cat M=3+k$ and~$\dim M=5+2k+l$.
Furthermore, if~$\pi$ is not free than~$M$ can be chosen
$4$-dimensional with~$\cat M=3$.
\end{theorem}

\begin{rem}
Our Theorem~\ref{t:realization} indicates that Theorem~\ref{11},
namely our generalization of \cite{GoGo}, in fact optimally clarifies
their result, to the extent that we identify the precise nature of the
effect of the fundamental group on the LS category of manifolds.
\end{rem}

The above results lead to the following questions:

\begin{question}
If~$M^4$ has free fundamental group, then we have the bound~$\cat M
\leq 3$.  Is it true that~$\cat M \leq 2$?
\end{question}

On the side of systolic category (see below), the following result is
immediate from the structure of the spaces $K(\Z,1)$ and $K(\Z,2)$,
cf.~Theorem~\ref{64}.

\begin{theorem}
An $n$-manifold with free fundamental group and torsion-free $H_2$ is
of systolic category at most $n-2$.
\end{theorem}

Partly motivated by this observation, and also by the similar bound
for the cuplength, we were led to the following question on the LS
side.

In an earlier version of the paper, we asked whether the inequality
$\cat M \le n-2$ holds for connected $n$-manifolds with free
fundamental group and $n>3$. Recently J. Strom~\cite{S2} proved that
this holds for all $n>4$, and even for $CW$-spaces, not only
manifolds.  Subsequently, we found a simple proof of this assertion
for manifolds, but again with the restriction $n>4$, see
\propref{p:n>4}.

\begin{question}
If~$\cat M=2$, is it true that~$\cat (M \setminus \{pt\})=1$?  A
direct proof would imply the main theorem trivially.
\end{question}

\begin{question}
Given integers~$m$ and~$n$, describe the fundamental groups of closed
manifolds~$M$ with~$\dim M=n$ and~$\cat M=m$.
\end{question}

Note that in the case~$m=n$, the fundamental group of~$M$ is of
cohomological dimension~$\ge n$, see e.g. \cite[Proposition
2.51]{CLOT}. Thus, we can ask when the converse holds.

\begin{question}
Given a finitely presented group~$\pi$ and an integer $n\geq 4$ such
that~$H^n(\pi)\not = 0$, when can one find a closed manifold~$M$
satisfying~$\pi_1(M)=\pi$ and $\dim M=\cat M=n?$ Note that
\propref{p:lense} shows that such a manifold~$M$ does not always
exist.
\end{question}

We apply \corref{12b} to prove that the systolic category of
a~$4$-manifold is a lower bound for its LS category.  The definition
of systolic category is reviewed in Section~\ref{upbound}.

\begin{theorem}
\label{12}
Every closed orientable~$4$-manifold~$M$ satisfies the
inequality~${\rm cat}_{\sys}(M) \leq \cat(M)$.
\end{theorem}

Note that the result on~$\cat$ applies to all
topological~$4$-manifolds.  Such a manifold is homotopy equivalent to
a finite~$CW$-complex.  The notion of~$\syscat$ makes sense for an
arbitrary finite~$CW$-complex, as the latter is homotopy equivalent to
a simplicial one, in which volumes and systoles can be defined.

We also prove the following lower bound for systolic category, which
is a weak analogue of Corollary~\ref{12b}, by exploiting a technique
based on Lescop's generalization of the Casson-Walker invariant
\cite{Les, KL}.  The Abel--Jacobi map of~$M$ is reviewed in
Section~\ref{six}.

\begin{theorem}
Let~$M$ be an~$n$-manifold satisfying~$b_1(M)=2$.  If the self-linking
class of a typical fiber of the Abel--Jacobi map is a nontrivial class
in~$H_{n-3}(M)$, then~$\syscat(M)\geq 3$.
\end{theorem}

The proof of the main theorem proceeds roughly as follows. If the
group~$\pi:=\pi_1(M)$ is not free, then by a result of J.~Stallings
and R.~Swan, the group~$\pi$ is of cohomological dimension at
least~$2$.  We then show that~$\pi$ carries a suitable
nontrivial~$2$-dimensional cohomology class~$u$ with twisted
coefficients, and of category weight~$2$.  Viewing~$M$ as a subspace
of~$K(\pi,1)$ that contains the~$2$-skeleton~$K(\pi,1)^{(2)}$, and
keeping in mind the fact that the~$2$-skeleton carries the fundamental
group, we conclude that the restriction (pullback) of~$u$ to~$M$ is
non-zero and also has category weight~$2$.  By Poincar\'e duality with
twisted coefficients, one can find a complementary~$(n-2)$-dimensional
cohomology class.  By a category weight version of the cuplength
argument, we therefore obtain a lower bound of~$3$ for~$\cat M$.

In Section~\ref{local}, we review the material on local coefficient
systems, a twisted version of Poincar\'e duality, and 2-dimensional
cohomology of non-free groups.  In Section~\ref{three}, we review the
notion of category weight.  In Section~\ref{four}, we prove the main
theorem. In Section~\ref{higher} we prove \theoref{t:realization}.
In Section~\ref{upbound}, we recall the notion of systolic
category, and prove that it provides a lower bound for the LS category
of a~$4$-manifold.  In Section~\ref{six}, we recall a 1983 result of
M.~Gromov's and apply it to obtain a lower bound of~$3$ for systolic
category for a class of manifolds defined by a condition of
non-trivial self-linking of a typical fiber of the Abel--Jacobi map.

\section{Cohomology with local coefficients}
\label{local}

A {\it local coefficient system}~$\A$ on a path
connected~$CW$-space~$X$ is a functor from the fundamental
groupoid~$\Gamma(X)$ of~$X$, to the category of abelian groups.  See
\cite{Ha}, \cite{Wh} for the definition and properties of local
coefficient systems.

In other words, an abelian group~$\A_x$ is assigned to each
point~$x\in X$, and for each path~$\ga$ joining~$x$ to~$y$, an
isomorphism~$\ga^*:\A_y \to \A_x$ is given.  Furthermore, paths that
are homotopic are required to yield the same isomorphism.

Let~$\pi=\pi_1(X)$, and let~$\zp$ be the group ring of~$\pi$.  Note
that all the groups~$A_x$ are isomorphic to a fixed group~$A$.  We
will refer to~$A$ as a {\it stalk\/} of~$\A$.

Given a map~$f: Y \to X$ and a local coefficient system~$\A$ on~$X$,
we define a local coefficient system on~$Y$, denoted~$f^*\A$, as
follows.  The map~$f$ yields a functor~$\Gamma(f): \Gamma(Y) \to
\Gamma(X)$, and we define~$f^*\A$ to be the functor~$\A\circ
\Gamma(f)$. Given a pair of coefficient systems~$\A$ and~$\B$, the
tensor product~$\A\otimes \B$ is defined by setting~$(\A\otimes
\B)_x=\A_x\otimes\B_x$.

\m \begin{ex}\label{ex:bundle} A useful example of a local coefficient
system is given by the following construction.  Given a fiber bundle
$p: E \to X$ over~$X$, set~$F_x=p^{-1}(x)$.  Then the family
$\{H_k(F_x)\}$ can be regarded a local coefficient system, see
\cite[Example 3, Ch. VI, \S 1]{Wh}.  An important special case is that
of an~$n$-manifold~$M$ and spherical tangent bundle~$p: E \to M$ with
fiber~$S^{n-1}$, yielding a local coefficient system~$\O$ with
$\O_x=H_{n-1}(S^{n-1}_x)\cong \Z$.  This local system is called the
{\it orientation sheaf\/} of~$M$.
\end{ex}

\begin{rem}
\label{22}
There is a bijection between local coefficients on~$X$ and
$\zp$-modules \cite[Ch. 1, Exercises F]{Sp}. If~$\A$ is a local
coefficient system with stalk~$A$, then the natural action of the
fundamental group on~$A$ turns~$A$ into a~$\Z[\pi]$-module.
Conversely, given a~$\Z[\pi]$-module~$A$, one can construct a local
coefficient system~$\scr L(A)$ such that induced~$\Z[\pi]$-module
structure on~$A$ coincides with the given one, cf.~\cite{Ha}.
\end{rem}

We recall the definition of the (co)homology groups with local
coefficients via modules \cite{Ha}:
\begin{equation}\label{cohomology}
H^k(X;\A)\cong H^k(\Hom_{\zp}(C_*(\wt X), A), \delta)
\end{equation}
and
\begin{equation}\label{homology}
H_k(X;\A)\cong H_k(A\otimes _{\zp}C_*(\wt X), 1\otimes \pa).
\end{equation}
Here~$(C_*(\wt X), \pa)$ is the chain complex of the universal
cover~$\wt X$ of~$X$,~$A$ is the stalk of the local coefficient
system~$\A$, and~$\delta$ is the coboundary operator.  Note that in
the tensor product we used the right~$\zp$ module structure on~$A$
defined via the standard rule~$ag=g^{-1}a$, for~$a\in A, g\in \pi$.

Recall that for~$CW$-complexes~$X$, there is a natural bijection
between equivalence classes of local coefficient systems and locally
constant sheaves on~$X$.  One can therefore define (co)homology with
local coefficients as the corresponding sheaf cohomology \cite{Br}.
In particular, we refer to \cite{Br} for the definition of the cup
product
\[
\cup: H^i(X;\A) \otimes H^j(X;\B) \to H^{i+j}(X; \A\otimes \B)
\]
and the cap product
\[
\cap: H_i(X;\A) \otimes H^j(X;\B) \to H_{i-j}(X; \A\otimes \B).
\]
A nice exposition of the cup  and the cap products in a slightly
different setting can be found in \cite{Bro}.  In particular, we have
the cap product
\[
H_k(X;\A)\otimes H^k(X;\B)\to H_0(X;\A\otimes \B)\cong
A\otimes_{\zp}B.
\]

\begin{prop}
\label{l:evaluation}
Given an integer~$k\ge 0$, there exists a local coefficient
system~$\B$ and a class~$v\in H^k(X;\B)$ such that, for every local
coefficient system~$\A$ and nonzero class~$a\in H_k(X;\A)$, we
have~$a\cap v\ne 0$.
\end{prop}

\begin{proof}
We convert the stalk of~$\A$ into a right~$\zp$-module~$A$ as above.
We use the isomorphisms \eqref{cohomology} and
\eqref{homology}. Consider the chain~$\zp$-complex
\[
\CD \dots @>>> C_{k+1}(\wt X) @>\pa_{k+1}>>C_k(\wt X) @>\pa_k>>
C_{k-1}(\wt X) @>>> \dots .
\endCD
\]
For the given~$k$, we set~$B: =C_k(\wt X)/\Im \pa_{k+1}$.  Let~$\B$ be
the corresponding local system on~$X$.  Thus, we obtain the exact
sequence of~$\zp$-modules
\[
\CD C_{k+1}(\wt X) @>\pa_{k+1}>>C_k(\wt X)@>f>>B\to 0.  \endCD
\]
Note that the epimorphism~$f$ can be regarded as a~$k$-cocycle with
values in~$\B$, since~$\delta f (x)=f\pa_{k+1}(x)=0$.  Let~$v:=[f]\in
H^k(X;\B)$ be the cohomology class of~$f$.  Now we prove that
$$
a\cap [f] \not=0.
$$

Since the tensor product is right exact, we obtain the diagram
\[
\CD A\otimes_{\zp}C_{k+1}(\wt X)
@>1\otimes\pa_{k+1}>>A\otimes_{\zp}C_k(\wt X)@>1\otimes
f>>A\otimes_{\zp}B @>>> 0\\ @. @. g @ VVV @.\\
@. @. A\otimes_{\zp}C_{k-1}(\wt X) \endCD
\]
where the row is exact.  The composition
\[
\CD A\otimes_{\zp}C_k(\wt X)@>1\otimes f>>A\otimes_{\zp}B @>g>>
A\otimes_{\zp}C_{k-1}(\wt X)
\endCD
\]
coincides with~$1\otimes \pa_k$. We represent the class~$a$ by a cycle
\[
z\in A\otimes_{\zp} C_k(\wt X).
\]
Since~$z\notin \Im(1\otimes \pa_{k+1})$, we conclude that
$$
(1\otimes f)(z)\ne 0\in A\otimes_{\zp}B=H_0(X;\A\otimes \B).
$$
Thus, for the cohomology class~$v$ of~$f$ we have~$a\cap v \ne 0$.
\end{proof}

Every closed connected~$n$-manifold~$M$ satisfies~$H_n(M; \O)\cong
\Z$.  A generator (one of two) of this group is called the {\it
fundamental class} of~$M$ and is denoted by~$[M]$.

\m One has the following generalization of the Poincar\'e duality
isomorphism.

\begin{thm}[\cite{Br}]\label{th:PD}
The homomorphism
\begin{equation}\label{eq:PD}
\Delta: H^i(M;\A)\to H_{n-i}(M;\O\otimes \A)
\end{equation}
defined by setting~$\Delta(a)=a\cap [M]$, is an isomorphism.
\end{thm}

In fact, in \cite{Br} there is the sheaf~$\O^{-1}$ at the right, but
for manifolds we have~$\O=\O^{-1}$.

Given a group~$\pi$ and a~$\Z[\pi]$-module~$A$, we denote by
$H^*(\pi;A)$ the cohomology of the group~$\pi$ with coefficients
in~$A$, see e.g.~\cite{Bro}.  Recall that~$H^i(\pi;A)=
H^i(K(\pi,1);\scr L(A))$, see Remark~\ref{22}.

\m Let~$F$ be a principal ideal domain and let~$\cd_F(\pi)$ denote the
cohomological
dimension of~$\pi$ over~$F$, i.e. the largest~$m$ such that there
exists an~$F[\pi]$-module~$A$ with~$H^m(\pi; A)\ne 0$.

\begin{thm}[\cite{Stal, Swan}]\label{th:free}
If~$\cd_{\Z}\pi\le 1$ then~$\pi$ is a free group.
\end{thm}

Wed need the following well-known fact.

\begin{lemma}\label{l:cd}
If~$\pi$ be a group with~$\cd_{\Z}\pi=q\geq 2$.
Then~$H^2(\pi;A)\ne 0$ for some~$\zp$-module~$A$.
\end{lemma}

\begin{proof}
Let~$0\to A'\to J\to A''\to 0$ be an exact sequence of~$\zp$-modules
with~$J$ injective. Then~$H^k(\pi;A')=H^{k-1}(\pi;A'')$ for $k>1$.
Since~$H^q(\pi;B)\ne 0$ for some~$B$, the proof can be completed by
an obvious induction.
\end{proof}

\begin{rem}
Let~$u\in H^1(\pi; I(\pi))$ be the Berstein-\v Svarc class described
in~\cite[Proposition 2.51]{CLOT}. If~$3\le \cd_{\Z}(\pi)=n<\infty$
then~$\cat(K(\pi,1))=n=\dim K(\pi,1)$ by \cite{EG}, and hence
$u^{\otimes n} \ne 0$ by ~\cite[Proposition 2.51]{CLOT}
(for~$n=\infty$ this means that~$u^{\otimes k}\ne 0$ for all~$k$). In
particular,~$H^2(\pi;I(\pi)\otimes I(\pi))\ne 0$. However, we do not
know if~$u\otimes u\ne 0$ in case~$\cd_{\Z}(\pi) =2$.
\end{rem}

\section{Category weight and lower bounds for~$\cat$}
\label{three}

\begin{definition}[\cite{BG,Fe,F}]
\label{def:cat-map}
Let~$f: X \to Y$ be a map of (locally contractible)~$CW$-spaces. The
\emph{Lusternik--Schnirelmann category of~$f$}, denoted~$\cat(f)$, is
defined to be the minimal integer~$k$ such that there exists an open
covering~$\{U_0, \ldots, U_k\}$ of~$X$ with the property that each of
the restrictions~$f|A_i\colon A_i \to Y$,~$i=0,1, \ldots, k$ is
null-homotopic.

The \emph{Lusternik--Schnirelmann category~$\cat X$ of a space~$X$} is
defined as the category~$\cat (1_X)$ of the identity map.
\end{definition}

\begin{definition}\label{def:swgt}
The \emph{category weight}~$\wgt(u)$ of a non-zero cohomology class~$u \in
H^*(X; \A)$ is defined as follows:
\begin{equation*}
\label{31}
\wgt(u)\ge k \Longleftrightarrow \{\gf^*(u)=0 {\rm\ for\ every\ } \gf\colon F
\to X
{\rm\ with\ } \cat(\gf) < k\}.
\end{equation*}
\end{definition}

\begin{rem}\label{rem:credits}\rm
E.~Fadell and S.~Husseini (see \cite{FH}) originally proposed the
notion of category weight.  In fact, they considered an invariant
similar to the~$\wgt$ of \eqref{31} (denoted in \cite{FH} by~$\cwgt$),
but where the defining maps~$\gf\colon F \to X$ were required to be
inclusions rather than general maps.  As a consequence,~$\cwgt$ is not
a homotopy invariant, and thus a delicate quantity in homotopy
calculations.  Yu.~Rudyak \cite{R1, R2} and J.~Strom \cite{S1}
suggested the homotopy invariant version of category weight as defined
in \defref{def:swgt}. Rudyak called it \emph{strict} category weight
(using the notation~$\swgt (u)$) and Strom called it \emph{essential}
category weight (using the notation~$E(u)$).  At the Mt.~Holyoke
conference in 2001, both creators agreed to adopt the notation~$\wgt$
and call it simply \emph{category weight}.
\end{rem}

\begin{prop}[\cite{R1,S1}]
\label{prop:swgtprops}
Category weight has the following properties.
\begin{enumerate}
\item~$1\le \wgt(u) \leq \cat(X)$, for all~$u \in \widetilde
H^*(X;\A), u\ne 0$.
\vskip3pt
\item For every~$f\colon Y \to X$ and~$u\in H^*(X;\A)$ with
$f^*(u)\not = 0$ we have
$\cat(f) \geq \wgt(u)$ and~$\wgt(f^*(u)) \geq \wgt(u)$.
\vskip3pt
\item For~$u\in H^*(X;\A)$ and~$v\in H^*(X;\B)$ we have
\[
\wgt(u\cup v) \geq \wgt(u) + \wgt(v).
\]  \vskip3pt
\item For every~$u \in H^s(K(\pi,1);\A)$,~$u\ne 0$, we have
$\wgt(u)\geq s$.  \vskip3pt
\end{enumerate}
\end{prop}

\begin{proof}
See \cite[\S 2.7 and Proposition 8.22]{CLOT}, the proofs in
loc. cit. can be easily adapted to local coefficient systems.
\end{proof}

\section{Manifolds of LS category~$2$}
\label{four}

In this section we prove that the fundamental group of a closed
connected manifold of LS category~$2$ is free.

\begin{thm}
\label{th:main}
Let~$M$ be a closed connected manifold of dimension~$n\ge 3$.  If the
group~$\pi:=\pi_1(M)$ is not free, then~$\cat M \geq 3$.
\end{thm}

\begin{proof}
By \theoref{th:free} and \lemref{l:cd}, there a local coefficient
system~$\A$ on~$K(\pi,1)$ such that~$H^2(K(\pi,1);\A)\ne 0$.  Choose a
non-zero element~$u\in H^2(K(\pi,1);\A)$. Let~$f: M \to K(\pi,1)$ be
the map that induces an isomorphism of fundamental groups, and let~$i:
K\to M$ be the inclusion of the~$2$-skeleton. (If~$M$ is not
triangulable, we take~$i$ to be any map of a~$2$-polyhedron that
induces an isomorphism of fundamental groups.)  Then
\[
(fi)^*: H^2(K(\pi,1);\A) \to H^2(K;(fi)^*\A)
\]
is a monomorphism.  In particular, we have~$f^*u\not=0$
in~$H^2(M;(f)^*\A)$.  Now consider the class
\[
a=[M]\cap f^*u\in H_{n-2}(M;\O^{-1}\otimes f^*\A).
\]
Then~$a\ne 0$ by \theoref{th:PD}.  Hence, by
Proposition~\ref{l:evaluation}, there exists a class~$v\in
H^{n-2}(M;\B)$ such that~$a\cap v \ne 0$. We claim that~$f^*u\cup v
\ne 0$. Indeed, one has
\[
[M]\cap (f^*u\cup v)=([M]\cap f^*u)\cap v =a\cap v\ne 0.
\]
Now,~$\wgt f^*u\ge 2$ by \propref{prop:swgtprops}, items (2)
\and~(4). Furthermore,~$\wgt(v)\geq 1$ by \propref{prop:swgtprops},
item~(1).  We therefore obtain the lower bound~$\wgt(f^*u\cup v)\ge 3$
by \propref{prop:swgtprops}, item (3).  Since~$f^*u\cup v\ne 0$, we
conclude that~$\cat M \ge 3$ by \propref{prop:swgtprops}, item (1).
\end{proof}

\begin{cor}
If~$M^n, n\geq 3$ is a closed manifold with~$\cat M\le 2$, then
$\pi_1(M)$ is a free group.
\end{cor}

The following Proposition is a special case of ~\cite[Corollary
2]{S2}. Here we give a relatively simple geometric proof.

\begin{prop}\label{p:n>4}
Let $M$ be a closed connected $n$-dimensional PL manifold, $n>4$, with
free fundamental group.  Then $\cat M \le n-2$.
\end{prop}

\begin{proof}
If $X$ is a 2-dimensional (connected) $CW$-complex with free
fundamental group then $\cat X\le 1$, see e.g.~\cite[Theorem
12.1]{KRS}. Hence, if $Y$ is a $k$-dimensional complex with free
fundamental group then $\cat Y\le k-1$ for $k>2$. Now, let $K$ be a
triangulation of $M$, and let $L$ be its dual triangulation.  Then
$M\setminus L^{(l)}$ is homotopy equivalent to $K^{(k)}$ whenever
$k+l+1=n$.  Hence, 
\[
\cat M\le \cat K^{(k)}+\cat L^{(l)}+1.
\]
Since $\pi_1(K)$ and $\pi_1(L)$ are free, we conclude that $\cat
K^{(k)}\le k-1$ and $\cat L^l\le l-1$ for $k,l>1$. Thus $\cat M\le k-1
+l-1 +1=n-2$.
\end{proof}

\section{Manifolds of higher LS category}
\label{higher}

\begin{definition}
\label{def:k-essent}
A~$CW$-space~$X$ is called~$k$-essential,~$k>1$ if for every
$CW$-complex structure on~$X$ there is no map
$f:X^{(k)}\to K(\pi,1)^{(k-1)}$ that induces an isomorphism of the
fundamental groups.
\end{definition}

\begin{thm}\label{th:k-essent}
For every closed~$k$-essential manifold~$M$ with~$\dim M>k$ we have
$\cat M\ge k+1$.
\end{thm}

\begin{proof}
Let~$M$ be~$k$-essential.

Let~$k=2$. If~$\cat M\le k$, then, by
\theoref{th:main},~$\pi_1(M)$ is free.
Hence there is a map~$f:M\to \vee S^1$ that induces an isomorphism
of the fundamental groups. Thus,~$M$ cannot be 2-essential.

Let~$k\ge 3$. Let~$K=K(\pi_1(M),1)$. Take a map~$f: M^{(k-1)} \to
K^{(k-1)}$ such that the restriction~$f|_{M^{(2)}}$ is the identity
homeomorphism of the 2-skeleta~$M^{(2)}$ and~$K^{(2)}$. We consider
the problem of extension of~$f$ to~$M$.

We claim that the first obstruction ~$o(f)\in H^k(M;E)$
(taken with coefficients in a local system~$E$ with the stalk
$\pi_{k-1}(K^{(k-1)})$) to the extension is not equal to zero.

Indeed, if~$o(f)=0$, then there exists a map~$\ov f:M^{(k)}\to
K^{(k-1)}$ which coincides with~$f$ on the~$(k-2)$-skeleton. The
map~$\ov f_*: \pi_1(M^{(k)})\to \pi_1(K^{(k-1)})$ can be regarded as
an endomorphism of~$\pi_1(M)$ that is identical on generators, and
therefore~$\ov f_*$ is an isomorphism. Hence,~$M$ is
not~$k$-essential.

Consider the commutative diagram
\[
\CD
M^{(k-1)} @>f>> K^{(k-1)} @>{\rm id}>> K^{(k-1)}\\
@ViVV @VjVV @.\\
M @>\tilde f>> K
\endCD
\]
where~$i$ and~$j$ are the inclusions of the skeleta.  Let~$\alpha$ be
the first obstruction to the extension of id to a map~$K \to
K^{(k-1)}$.  By commutativity of the above diagram, we
have~$o(f)=\tilde f^*(\alpha)$.  Now, asserting as in the proof of
\theoref{th:main}, we get that~$\tilde f^*(\alpha)\cup v\ne 0$ for
some~$v$ with~$\dim v=\dim M-k$.  Since~$\dim M>k$, we conclude
that~$\dim v\ge 1$ and thus~$\cat M\ge k+1$.
\end{proof}

\begin{rem}
If a closed manifold~$M^n$ is~$n$-essential then~$\cat M=n$, see
e.g.~\cite{KR1} and~\cite[Theorem 12.5.2]{SGT}.
\end{rem}

\begin{prop}\label{p:nonfree}
For every non-free finitely presented group~$\pi$, there exists a
closed~$4$-dimensional manifold~$M$ with fundamental group~$\pi$
and~$\cat M=3$.
\end{prop}

\begin{proof}
Take an embedding of a~$2$-skeleton of~$K(\pi,1)$ in~$\R^5$ and
let~$M$ be the boundary of the regular neighborhood of this skeleton.
Then, clearly,~$\pi_1(M)=\pi$. Furthermore,~$M$ admits a retraction
onto its~$2$-skeleton. Therefore~$M$ is not~$4$-essential, and hence
$\cat M = 3$.
\end{proof}

Let~$M_f$ be the mapping cylinder of~$f:X\to Y$.  We use the notation
$\pi_*(f)=\pi_*(M_f,X)$.  Then~$\pi_i(f)=0$ for~$i\le n$ amounts to
saying that it induces isomorphisms~$f_*:\pi_i(X_1)\to \pi_i(Y_1)$ for
$i\le n$ and an epimorphism in dimension~$n+1$. Similar notation
$H_*(f)=H_*(Mf,X)$ we use for homology.

\begin{lemma}
\label{l:join} Let~$f_j:X_j\to Y_j$ be a family of maps of~$CW$
spaces such that~$H_i(f_j)=0$ for~$i<n_j$. Then~$ H_i(f_1\wedge
\cdots \wedge f_s)=0$ for~$i\le \min\{n_j\}$.
\end{lemma}

\begin{proof}
Note that~$M(f_1\wedge \cdots \wedge f_s)\cong Y_1\wedge \cdots \wedge
Y_s\cong M(f_1)\wedge \cdots \wedge M(f_s)$. Now, by using the
K\"unneth formula and considering the homology exact sequence of the
pair~$(M(f_1)\wedge\cdots \wedge M(f_s), X_1\wedge \cdots \wedge X_s)$
we get the result.
\end{proof}

\begin{prop}\label{p:join}
Let~$f_j:X_j\to Y_j$,~$3\le j\le s$ be a family of maps of~$CW$
spaces such that~$\pi_i(f_j)=0$ for~$i<n_j$.  Then the joins
satisfy
\[
\pi_k(f_1\ast f_2\ast\dots\ast f_s)=0
\]
for~$k\le\min\{n_j\}+s-1$.
\end{prop}

\begin{proof}
By the version of the Relative Hurewicz Theorem for non-simply
connected~$X_j$ \cite[Theorem 4.37]{Ha}, we obtain~$H_i(f_j)=0$ for
$i<n_j$.  By \lemref{l:join} we obtain that~$H_k(f_1\wedge\cdots\wedge
f_s)=0$ for~$k\le\min\{n_j\}$. Since the join~$A_1\ast\dots\ast A_s$ is homotopy equivalent to the iterated suspension~$\Sigma^{s-1}(A_1\wedge\dots\wedge A_s)$ over the smash product,
we conclude that~$H_k(f_1\ast\dots\ast f_s)=0$ for~$k\le\min\{n_j\}+s-1$.
Since~$X_1\ast\dots\ast X_s$ is
simply connected for~$s\ge 3$, by the standard Relative Hurewicz
Theorem we obtain that~$\pi_k(f_1\ast\dots \ast f_s)=0$ for
$k\le \min\{n_j\}+s-1$.
\end{proof}

Given two maps~$f:Y_1\to X$ and~$g:Y_2\to X$, we set
\[
Z=\{(y_1,y_2,t)\in Y_1\ast Y_2\mid f(y_1)=g(y_2)\}
\]
and define the {\em fiberwise join}, or {\em join over~$X$} of~$f$
and~$g$ as the map
\[
f{\ast_X}g:Z\to X,\quad (f{\ast_X}g)(y_1,y_2, t)=f(y_1)
\]
Let~$p_0^X:PX\to X$ be the Serre path fibration. This means that
$PX$ is the space of paths on~$X$ that start at the base point of the
pointed space~$X$, and~$p_0(\alpha)=\alpha(1)$.  We denote by
$p_n^X;G_n(X)\to X$ the~$n$-fold fiberwise join of~$p_0$.

The proof of the following theorem can be found in \cite{CLOT}.

\begin{thm}[Ganea, \v Svarc]
\label{t:ganea}
For a~$CW$-space~$X$,~$\cat(X)\le n$ if and only if there exists a
section of~$p_n:G_n(X)\to X$.
\end{thm}

\begin{prop}\label{p:sum}
The connected sum~$S^k\times S^l\#\cdots \#S^k\times S^l$ is a space
of LS-category~$2$.
\end{prop}

\begin{proof}
This can be deduced from a general result of K. Hardy~\cite{H} because
the connected sum of two manifolds can be regarded as the double
mapping cylinder. Alternatively, one can note that, after removing a
point, the manifold on hand is homotopy equivalent to the wedge of
spheres.
\end{proof}

\begin{thm}
\label{t:high}
For every finitely presented group~$\pi$ and~$n\ge 5$, there is a
closed~$n$-manifold~$M$ of LS-category~$3$ with~$\pi_1(M)=\pi$.
\end{thm}

\begin{proof}
If the group~$\pi$ is the free group of rank~$s$, we let~$M'$ be the
$k$-fold connected sum~$S^1\times S^{2} \# \cdots \#S^1\times S^{2}$.
Then~$M'$ is a closed~$3$-manifold of LS category~$2$ with
$\pi_1(M')=F_s$.  Then the product manifold~$M=M'\times S^{n-3}$ has
cuplength~$3$ and is therefore the desired manifold.

Now assume that the group~$\pi$ is not free.  We fix a presentation
of~$\pi$ with~$s$ generators and~$r$ relators.  Let~$M'$ be
the~$k$-fold connected sum~$S^1\times S^{n-1} \# \cdots \#S^1\times
S^{n-1}$.  Then~$M'$ is a closed~$n$-manifold of the category~$2$ with
$\pi_1(M')=F_s$.  For every relator~$w$ we fix a nicely imbedded
circle~$S^1_w\subset M'$ such that~$S_w^{-1}\cap S_v^{-1}=\emptyset$
for~$w\ne v$.  Then we perform the surgery on these circles to obtain
a manifold~$M$. Clearly,~$\pi_1(M)=\pi$. We show that~$\cat(M)\le 3$,
and so~$\cat M =3$ by \theoref{th:main}.

As usual, the surgery process yields an~$(n+1)$-manifold~$X$ with
$\partial X=M\sqcup M'$.  Here~$X$ is the space obtained from
$M'\times I$ by attaching handles~$D^2\times D^{n-1}$ of index 2 to
$M'\times 1$ along the above circles. We note that~$\cat(X)\le 3$.

On the other hand, by duality,~$X$ can be obtained from~$M \times I$
by attaching handles of index~$n-1$ to the boundary component of~$M
\times I$.  In particular, the inclusion~$f: M \to X$ induces an
isomorphism of the homotopy groups of dimension~$\le n-3$ and an
epimorphism in dimension~$n-2$.  Hence~$\Omega f:\Omega M\to\Omega X$
induces isomorphisms in dimensions~$\le n-4$ and an epimorphism in
dimension~$n-3$.  Thus,~$\pi_i(\Omega f)=0$ for~$i\le n-4$.

In order to prove that~$\cat M\le 3$ it suffices to show that the
Ganea-\v Svarc fibration~$p_3:G_3(M)\to M$ has a section.  Consider
the commutative diagram
\[
\CD
G_3M @>q> >Z @>f'>> G_3(X)\\
@Vp_M^3VV @Vp'VV @ VVp_3^XV\\
M @= M @>f>> X\\
\endCD
\]
where the right-hand square is the pull-back diagram
and~$f'q=G_3(f)$. Note that~$q$ is uniquely determined.
Since~$\cat(X)\le 3$, by \theoref{t:ganea} there is a section~$s:X\to
G_3(X)$.  It defines a section~$s':M\to Z$ of~$p'$. It suffices to
show that the map~$s':M\to Z$ admits a homotopy lifting~$h: M \to
G_3M$ with respect to~$q$, i.e. the map~$h$ with ~$qh\cong
s'$. Indeed, we have
\[
p_M^3h=p'qh\cong p's'=1_M
\]
and so~$h$ is a homotopy section of~$p_3^M$. Since the latter is a
Serre fibration, the homotopy lifting property yields an actual
section.

Let~$F_1$ and~$F_2$ be the fibers of fibrations~$p_3^M$ and~$p'$,
respectively. Consider the commutative diagram generated by the homotopy exact sequences
of the Serre fibrations~$p_3^M$ and~$p'$:
\[
\CD
\pi_i(F_1) @>>> \pi_i(G_3(M)) @>(p_3^M)_*>> \pi_i(M) @>>>
\pi_{i-1}(F_1) @>>>\cdots\\
 @VV\phi_*V @ VVq_*V @VV=V @VV\phi_*V @.\\
\pi_i(F_2) @>>> \pi_i(Z) @>(p')_*>> \pi_i(M) @>>>\pi_{i-1}(F_2)@>>>
\cdots .\\
\endCD
\]
Note that we have
\[
\phi=\Omega (f)\ast \Omega(f)\ast\Omega(f)\ast \Omega(f).
\]

By \propref{p:join} and since~$\pi_i(\Omega f)=0$ for~$i\le n-5$, we
conclude that~$\pi_i(\phi)=0$ for~$i\le n-4+3=n-1$.  Hence~$\phi$
induces an isomorphism of the homotopy groups of dimensions~$\le n-1$
and an epimorphism in dimension~$n$.  By the Five Lemma we obtain
that~$q_*$ is an isomorphism in dimensions~$\le n-1$ and an
epimorphism in dimension~$n$. Hence the homotopy fiber of~$q$
is~$(n-1)$-connected. Since~$\dim M=n$, the map~$s'$ admits a homotopy
lifting~$h:M\to G_3(M)$.
\end{proof}

\begin{cor}
\label{c:real}
Given a finitely presented group~$\pi$ and non-negative integer
numbers~$k, l$ there exists a closed manifold~$M$ such
that~$\pi_1(M)=\pi$, while~$\cat M=3+k$ and~$\dim M=5+2k+l$.
\end{cor}

\begin{proof}
By \theoref{t:high}, there exists a manifold~$N$ such that
$\pi_1(M)=\pi$,~$\cat M=3$ and~$\dim M=5+l$. Moreover, this manifold
$N$ possesses a detecting element, i.e. a cohomology class whose
category weight is equal to~$\cat N=3$. For~$\pi$ free this follows
since the cuplength of~$N$ is equal to 3, for other groups we have the
detecting element~$f^*u\cup v$ constructed in the proof of
\theoref{th:main}. If a space~$X$ possesses a detecting element then,
for every~$k>0$, we have~$\cat (X\times S^k) =\cat X+1$ and~$X \times
S$ possesses a detecting element, \cite{R2}. Now, the manifold
$M:=N\times (S^2)^k$ is the desired manifold.
\end{proof}

Generally, we have a question about relations between the category,
the dimension, and the fundamental group of a closed manifold.The
following proposition shows that the situation quite intricate.

\begin{prop}
\label{p:lense}
Let $p$ be an odd prime. Then there exists a closed $(2n+1)$-manifold
with $\cat M=\dim M$ and $\pi_1(M)=\Zp$, but there are no closed
$2n$-manifolds with $\cat M=\dim M$ and $\pi_1(M)=\Zp$.
\end{prop}

\begin{proof}
An example of $(2n+1)$-manifold is the quotient space $S^{2n+1}/\Zp$
with respect to a free $\Zp$-action on $S^{2n+1}$. Now, given a
$2n$-manifold with $\pi_1(M)=\Zp$, consider a map $f:M \to K(\Zp,1)$
that induces an isomorphism of fundamental groups. Since
$H_{2n}(K(\Zp,1))=0$, it follows from the obstruction theory and
Poincar\'e duality that $f$ can be deformed into the $(2n-1)$-skeleton
of $K(\Zp,1)$, cf.~\cite[Section 8]{Bab1}. Hence, $M$ is inessential,
and thus $\cat M <2n$ \cite{KR1}.
\end{proof}

\section{Upper bound for systolic category}
\label{upbound}

In this section, we recall the definition of systolic
category~$\syscat$, and prove an upper bound for~$\syscat$ of a smooth
manifold with free fundamental group.  Combined with the result of the
previous section, we thus obtain a proof of \theoref{12}.

The main idea behind the definition of systolic category is to bound
the total (top-dimensional) volume from below by a product of
lower-dimensional systolic information, in the following precise
sense, cf.~\eqref{dd}.

\begin{definition}
\label{21b}
\label{sys}
Let~$X$ be a Riemannian manifold with the Riemannian metric~$\gmetric$.
Given~$k\in \N$,~$k>1$, we set
\[
\ \sys_{k}(X, \gmetric)= \inf\{\sysh_k(X, \gmetric;\Z[\regular]),
\sysh_k(X, \gmetric;\Z_2[\regular]),\stsys_k(X, \gmetric) \},
\]
where the infimum is over all groups~$B$ of regular covering spaces
of~$X$.  Note that a systole of~$X$ with coefficients in the group
ring of~$B$ is by definition the systole of the corresponding covering
space~$\ov X \to X$ with deck group~$B$.  Here~$\sysh_k$ is the
homology~$k$-systole, while~$\Z[\regular]$ is the~$\Z$-group ring
of~$\regular$ and similarly for~$\Z_2$.  The invariant~$\stsys_k$ is
the stable~$k$-systole.  See \cite{KR1} and \cite[Chapter~12]{SGT} for
more detailed definitions.  Furthermore, we define
$$
\sys_{1}(X, \gmetric)=\min \{\pisys_1(X, \gmetric),\stsys_1(X,
\gmetric)\}.
$$
\end{definition}

Note that the systolic invariants thus defined are positive (or
infinite), see \cite{KR2}.

Let~$X$ be an~$n$-dimensional polyhedron, and let~$d\geq 1$ be an
integer.  Consider a partition
\begin{equation}
\label{2.1}
n= k_1 + \ldots + k_d,
\end{equation}
where~$k_i\geq 1$ for all~$i=1,\ldots, d$.  We will consider
scale-invariant inequalities ``of length~$d$'' of the following type:
\begin{equation}
\label{dd}
\sys_{k_1}(\gmetric) \sys_{k_2}(\gmetric) \ldots \sys_{k_d}(\gmetric)
\leq C(X) \vol_n(\gmetric),
\end{equation}
satisfied by all metrics~$\gmetric$ on~$X$, where the constant~$C(X)$
is expected to depend only on the topological type of~$X$, but not on
the metric~$\gmetric$.  Here the quantity~$\sys_k$ denotes the infimum
of all systolic invariants in dimension~$k$, as defined above.

\begin{definition}
\label{syscat}
Systolic category of~$X$, denoted~$\syscat(X)$, is the largest
integer~$d$ such that there exists a partition~\eqref{2.1} with
\[
\prod\limits^{d}_{i=1} \sys_{k_i}(X,\gmetric) \leq C(X) \vol_n(X,\gmetric)
\]
for all metrics~$\gmetric$ on~$X$.  If no such partition and
inequality exist, we define systolic category to be zero.
\end{definition}

In particular, we have~$\syscat X \le \dim X$.

\begin{rem}
Clearly, systolic category equals~$1$ if and only if the polyhedron
possesses an~$n$-dimensional homology class, but the volume cannot be
bounded from below by products of systoles of positive codimension.
Systolic category vanishes if~$X$ is contractible.
\end{rem}

\begin{thm}
\label{64}
Let~$M$ be an~$n$-manifold,~$n\geq 4$, with~$H_2(M)$ torsion-free.
Suppose the fundamental group of~$M$ is free.  Then its systolic
category is at most~$n-2$.
\end{thm}

\begin{proof}
By hypothesis, we have~$H_2(M)=\Z^{b_2}$.  Consider a map
\[
M\to K(\pi,1)\times K(\Z^{b_2},2),
\]
inducing an isomorphism of fundamental groups, as well as isomorphism
of~$2$-dimensional homology.  Here the first factor is a wedge of
circles, while the second is a product of~$b_2(M)$ copies of~$\C
{\mathbb P}^\infty$.  We work with the product cell structure.  We may
assume that the image of~$M$ lies in the~$n$-skeleton.  We consider
separately two cases according to the parity of the dimension.

If~$n$ is even, then all cells involving circles from the first factor
have positive codimension.  Since the map in homology is injective, we
can apply pullback arguments for metrics as in \cite{e7}.  Therefore
we can obtain metrics with fixed volume and fixed~$\sys_2$, but with
arbitrarily large~$\sys_1$.  In more detail, the pullback arguments
for metrics originate in I.~Babenko's 1992 paper translated in
\cite{Bab1}, where he treats the case of the~$1$-systoles.  In
\cite{Ba06}, he describes the argument for the stable~$1$-systole.  In
\cite{e7}, we present an argument that treats the case of the
stable~$k$-systole, involving a suitable application of the coarea
formula and an isoperimetric inequality.

This precludes the possibility of a lower bound for the total volume
corresponding to a partition of type~\eqref{2.1} satisfying
\[
\max_i k_i =2,
\]
i.e.~involving only the~$1$-systole and the~$2$-systole.  Thus any
lower bound must involve a~$k$-systole with~$k\geq 3$, and therefore
we obtain~$\syscat \leq n-2$, proving the theorem for even~$n$.

If~$n\geq 5$ is odd, each top dimensional cell is a copy of~$S^1\times
\C {\mathbb P}^{(n-1)/2}$.  By considering the product metric of a
circle of length~$L\to \infty$ with a fixed metric on~$\C {\mathbb
P}^{(n-1)/2}$, we see that both the volume and the~$1$-systole grow
linearly in~$L$.  Therefore a product of the form
\[
(\sys_1)^a (\sys_2)^b
\]
can only be a lower bound for the volume if the first exponent
satisfies~$a=1$.  Thus a factor of either~$(\sys_2)^2$ or~$\sys_k$
with~$k\geq 3$ must be involved in \eqref{dd}, and again we
obtain~$\syscat \leq n-2$.
\end{proof}

\begin{cor}
The systolic category of an orientable~$4$-manifold with free
fundamental group is at most~$2$.
\end{cor}

\begin{proof}
If the fundamental group is free, then by the universal coefficient
theorem,~$2$-dimensional cohomology is torsion-free.  By~$\Z$-Poincar\'e duality, the same is true of homology, and we apply the previous
theorem.
\end{proof}

\begin{proof}[Proof of Theorem~$\ref{12}$]
Let~$M$ be an orientable~$4$-manifold with free fundamental group.
The possible inequalities of type \eqref{dd} in dimension~$4$
correspond to one of the three partitions~$4=1+3$,~$4=2+2$,
and~$4=1+1+2$.  The latter is ruled out by the previous theorem.  In
either of the first two cases, we have~$\syscat=2$.  Since
a~$4$-manifold with~$\syscat=4$ must be essential with suitable
coefficients~\cite{Bab1, Gr1, KR1}, it follows that~$\syscat \leq
\cat$ for all orientable~$4$-manifolds.
\end{proof}

\section{Self-linking of fibers and a lower bound for~$\syscat$}
\label{six}

There are three main constructions for obtaining systolic lower bounds
for the total volume of a closed manifold~$M$.  All three originate is
Gromov's 1983 Filling paper~\cite{Gr1}, and can be summarized as
follows.

\begin{enumerate}
\item
Gromov's inequality for the homotopy~$1$-systole of an essential
manifold~$M$, see \cite{We, Gu, Bru2} and \cite[p.~97]{SGT}.
\item
Gromov's stable systolic inequality (treated in more detail in
\cite{BK1, BK2}) corresponding to a cup product decomposition of the
rational fundamental cohomology class of~$M$.
\item
A construction using the Abel--Jacobi map to the Jacobi torus of~$M$
(sometimes called the dual torus), also based on a theorem of Gromov
(elaborated in \cite{IK, BCIK2}).
\end{enumerate}

Let us describe the last construction in more detail.  Let~$M$ be a
connected~$n$-manifold.  Let~$b=b_1(M)$.  Let
\[
\T^b := H_1(M;\R)/H_1(M;\Z)_\R
\]
be its Jacobi torus (sometimes called the dual torus).  A natural
metric on the Jacobi torus of a Riemannian manifold is defined by the
stable norm, see \cite[p.~94]{SGT}.

The Abel--Jacobi map~$\AJ_M: M\to \T^b$ is discussed in \cite{Li, BK2},
cf.~\cite[p.~139]{SGT}.  A typical fiber~$F_M\subset M$ (i.e.~inverse
image of a generic point) of~$\AJ_M$ is a smooth
imbedded~$(n-b)$-submanifold (varying in type as a function of the
point of~$\T^b$).  Our starting point is the following observation of
Gromov's \cite[Theorem~7.5.B]{Gr1}, elaborated in \cite{IK}.

\begin{theorem}[M.~Gromov]
\label{61}
If the homology class~$\fmanifold\in H_{n-b}(M)$ of the lift of~$F_M$
to the maximal free abelian cover of~$M$ is nonzero, then the total
volume of~$M$ admits a lower bound in terms of the product of the
volume of the Jacobi torus and the infimum of areas of cycles
representing~$\fmanifold$.
\end{theorem}

\begin{prop}\label{p:fm}
If a typical fiber of the Abel--Jacobi map represents a nontrivial
$(n-b)$-dimensional homology class in~$M$, then systolic category
satisfies~$\syscat(M)\geq b+1$.
\end{prop}

\begin{proof}
If the fiber class is nonzero, then the Abel--Jacobi map is
necessarily surjective in the set-theoretic sense.  One then applies
the technique of Gromov's proof of Theorem~\ref{61}, cf.~\cite{IK},
combined with a lower bound for the volume of the Jacobi torus in
terms of the~$b$-th power of the stable~$1$-systole, to obtain a
systolic lower bound for the total volume corresponding to the
partition
\[
n=1+1+\cdots+1+(n-b),
\]
where the summand~$1$ occurs~$b$ times.
\end{proof}

Our goal is to describe a sufficient condition for applying Gromov's
theorem, so as to obtain such a lower bound in the case when the fiber
class in~$M$ vanishes.  From now on we assume that~$M$ is orientable,
has dimension~$n$, and~$b_1(M)=2$.  Let~$\{\alpha,\beta \} \subset
H^1(M)$ be a basis of~$H^1(M)$.  Let~$F_M$ be a typical fiber of the
Abel-Jacobi map.  It is easy to see that~$[F_M]$ is Poincar\'e dual to
the cup product~$\ga\cup \gb$.  Thus, if~$\ga\cup\gb\ne 0$
then~$\syscat M\ge 3$ by \propref{p:fm}.  If~$\ga\cup \gb=0$ then the
Massey product~$\la\ga,\ga,\gb\ra$ is defined and has zero
indeterminacy.

\begin{thm}\label{t:massey}
If~$\la\ga,\ga,\gb\ra\gb \ne 0$ then~$\syscat M \ge 3$.
\end{thm}

Note that also~$\cat M \ge 3$ if~$\la\ga,\ga,\gb\ra\ne 0$, since~$\la
\ga, \ga, \gb \ra$ has category weight~$2$~\cite{R2}.

To prove the theorem, we reformulate it in the dual homology language.

\begin{definition}
 Let~$F=F_M \subset M$ be an oriented typical fiber.  Assume~$[F]=0\in
H_{n-2}(M)$.  Choose an~$(n-2)$-chain~$X$ with~$\partial X = F$.
Consider another regular fiber~$F'\subset M$.  The oriented
intersection~$X\cap F'$ defines a class
\[
\ell_M(F_M, F_M) \subset H_{n-3}(M),
\]
which will be referred to as the {\em self-linking class\/} of a
typical fiber of~$\AJ_M$.
\end{definition}

The following lemma asserts, in particular, that the self-linking
class is well-defined, at least up to sign.

\begin{lemma}
The class~$\ell_M(F_M, F_M)$ is dual, up to sign, to the cohomology
class~$\la \alpha,\alpha,\beta \ra \cup\beta \in H^3(M)$.
\end{lemma}

\begin{proof}
The classes~$\alpha, \beta$ are Poincar\'e dual to hypersurfaces~$A, B
\subset M$ obtained as the inverse images under~$\AJ_M$ of a
pair~$\{u,v\}$ of loops defining a generating set for~$\T^2$. Clearly,
the intersection~$A\cap B \subset M$ is a typical fiber
\[
F_M=A\cap B
\]
of the Abel-Jacobi map (namely, inverse image of the point~$u\cap v\in
\T^2$). Then another regular fiber~$F'$ can be represented as~$A'\cap
B'$ where, say, the set~$A'$ is the inverse image of a loop~$u'$
``parallel" to~$u$. Then~$A'\cap X$ is a cycle, since~$\partial
(A'\cap X)=A'\cap A \cap B=\emptyset$. Moreover, it is easy to see
that the homology class~$[A'\cap X]$ is dual to the Massey
product~$\la \ga,\ga, \gb\ra$. (We take a representative~$a$ of~$\ga$
such that~$a\cup a=0$.) Now, since~$F'=A'\cap B'$, we conclude
that~$[F'\cap X]$ is dual, up to sign, to~$\la \ga,\ga,\gb\ra\cup
\gb$.
\end{proof}

\begin{remark}
In the case of~$3$-manifolds with Betti number~$2$, the non-vanishing
of the self-linking number is equivalent to the non-vanishing of
C.~Lescop's generalization~$\lambda$ of the Casson-Walker invariant,
cf.~\cite{Les}.  See T.~Cochran and J.~Masters~\cite{CM} for
generalizations.
\end{remark}

\m Now \theoref{t:massey} will follow from \theoref{t:self} below.

\begin{theorem}\label{t:self}
If the self-linking class in~$H_{n-3}(M)$ is non-trivial,
then~$\syscat(M)\geq 3$.
\end{theorem}

The theorem is immediate from the proposition below.  If the fiber
class in~$M$ of the Abel--Jacobi map vanishes, one can define the
self-linking of a typical fiber, and proceed as follows.

\begin{prop}
\label{propal}
The non-vanishing of the self-linking of a typical fiber
$\AJ_M^{-1}(p)$ of~$\AJ_M: M \to \T^2$ is a sufficient condition for
the non-vanishing of the fiber class~$\fmanifold$ in the maximal free
abelian cover~$\manbar$ of~$M$.
\end{prop}

\begin{proof}
The argument is modeled on the one found in \cite{KL} in the case of
3-manifolds, and due to A. Marin (see also \cite[p.~165-166]{SGT}).

Consider the pullback diagram
\[
\CD \manbar @>\ov{\AJ}_M >> \R^2\\ @VpVV @VVV\\ M @>\AJ_M >> \T^2
\endCD
\]
where~$\AJ_M$ is the Abel--Jacobi map and the right-hand map is the
universal cover of the torus.  Take~$x,y\in \R^2$.  Let~$\ov
F_x=\overline{\AJ}_M^{-1}(x)$ and~$\ov F_y= \overline{\AJ}_M^{-1}(y)$
be lifts of the corresponding fibers~$F_x, F_y \subset M$.  Choose a
properly imbedded ray~$r_y\subset \R^2$ joining the point~$y\in \R^2$
to infinity while avoiding~$x$ (as well as its~$\Z^2$-translates), and
consider the complete hypersurface
\[
S = \overline{\AJ}^{-1}(r_y) \subset \manbar
\]
with~$\partial S = \ov F_y$.  We have~$S\cap g. \ov F_x = \emptyset$
for all~$g\in G$, where~$G=\Z^2$ denotes the group of deck
transformations of the covering~$p: \manbar \to M$.

We will prove the contrapositive.  Namely, the vanishing of the class
of the lift of the fiber implies the vanishing of the self-linking
class.  If the surface~$\ov F_x$ is zero-homologous in~$\manbar$, we
can choose a compact hypersurface~$\surface \subset \manbar$ with
\[
\partial \surface = \ov F_x.
\]
The linking~$\ell_M(F_x, F_y)$ in~$M$ can therefore be computed as the
oriented intersection
\begin{equation}
\label{1321}
\begin{aligned}
\ell_M(F_x, F_y) & = p(\surface) \cap F_y \\ & = \sum _{g\in
G} g. \surface \cap \ov F_y \\ & = \sum _{g\in G} \partial \left(
g. \surface \cap S \right) \\ & \sim 0,
\end{aligned}
\end{equation}
where all sums and intersections are oriented, and the intersection
surface~$\surface \cap S$ is compact by construction.
\end{proof}

To summarize, if the lift of a typical fiber to the maximal free
abelian cover of~$M^n$ with~$b_1(M)=2$ defines a nonzero class, then
one obtains the lower bound~$\syscat(M)\geq 3$, due to the existence
of a suitable systolic inequality corresponding to the partition
\[
n=1+1+(n-2)
\]
as in \eqref{2.1}, by applying Gromov's Theorem~\ref{61}.

\end{document}